\documentclass{amsart}
\usepackage[mathscr]{eucal}

\usepackage{hyperref}

\allowdisplaybreaks[1]

\newtheorem{thrm}{Theorem}[section]
\newtheorem{lem}[thrm]{Lemma}

\newtheorem{cor}[thrm]{Corollary}
\theoremstyle{definition}
\newtheorem{definition}[thrm]{Definition}

\numberwithin{equation}{section}

\newtheorem{example}[thrm]{Example}
\newtheorem{problem}[thrm]{Problem}

\newcommand{\labeq}[1]{\label{eq:#1}}
\newcommand{\refeq}[1]{(\ref{eq:#1})}
\newcommand{\labt}[1]{\label{thm:#1}}
\newcommand{\reft}[1]{Theorem~\ref{thm:#1}}

\newcommand{\labl}[1]{\label{lemma:#1}}
\newcommand{\refl}[1]{Lemma~\ref{lemma:#1}}
\newcommand{\labd}[1]{\label{definition:#1}}

\newcommand{\formalsum}{\sum_{n=1}^{\infty} \frac {E_n} {q_1 q_2 \ldots q_n}}

\newcommand{\dimh}[1]{\hbox{$\dim_{\hbox{H}}$}\left( #1\right)}

\newcommand{\q}[1]{q_1 \cdots q_{ #1 }}

\newcommand{\NN}{\mathbb{N}_2^{\mathbb{N}}}
\newcommand{\NNC}{\NN \times \NN}

\newcommand{\wrt}[1]{\hbox{ w.r.t. }#1}

\newcommand{\ppq}{\psi_{P,Q}}

\newcommand{\floor}[1]{\left\lfloor #1 \right\rfloor} 
 
\newcommand{\br}[1]{\left\{ #1 \right\}}

\newcommand{\pr}[1]{\left( #1 \right)}

\newcommand{\NQ}{\mathscr{N}(Q)}
\newcommand{\N}[1]{\mathscr{N}( #1 )}
\newcommand{\Nk}[2]{\mathscr{N}_{#2}( #1 )} 
\newcommand{\DNQ}{\mathscr{DN}(Q)}
\newcommand{\DN}[1]{\mathscr{DN}( #1 )} 
\newcommand{\RNQ}{\mathscr{RN}(Q)}

\newcommand{\RDN}{\RNQ \cap \DNQ \backslash \NQ}

\newcommand{\blank}[1]{ }

\thanks{Research of the authors is partially supported by the U.S. NSF grant DMS-0943870.  We would like to thank Samuel Roth for posing the problem that led to \reft{main} to the second author  at the 2012 RTG conference: Logic, Dynamics and Their Interactions, with a Celebration of the Work of Dan Mauldin in Denton, Texas.  He asked if it is true that $x \in \NQ \cap \DNQ$ implies that $nx \in \NQ$ for all natural numbers $n$.  We thank Martin Sleziak for pointing us in a direction that led to the paper \cite{RajagopalWeightedDensity}. This paper helped us prove a stronger version of \reft{main}.}

\author[D. Airey]{Dylan Airey}
\address[D. Airey]{
Department of Mathematics, University of Texas at Austin, 2515 Speedway, Austin, TX 78712-1202, USA}
\email{dylan.airey@utexas.edu}

\author[B. Mance]{Bill Mance}
\address[B. Mance]{Department of Mathematics, University of North Texas, General Academics Building 435, 1155 Union Circle,  \#311430, Denton, TX 76203-5017, USA}
\email{mance@unt.edu}

\begin{document}

\title[Normality preserving operations part I]{Normality preserving operations for Cantor series expansions and associated fractals part I}

\begin{abstract}
It is well known that rational multiplication preserves normality in base $b$.  We study related normality preserving operations for the $Q$-Cantor series expansions.  In particular, we show that while integer multiplication preserves $Q$-distribution normality, it fails to preserve $Q$-normality in a particularly strong manner.  We also show that $Q$-distribution normality is not preserved by non-integer rational multiplication on a set of zero measure and full Hausdorff dimension.
\end{abstract}

\maketitle

\section{Introduction}
Let $\N{b}$ be the set of numbers normal in base $b$ and let $f$ be a  function from $\mathbb{R}$ to $\mathbb{R}$.  We say that $f$ {\it preserves $b$-normality} if $f(\N{b}) \subseteq \N{b}$.  We can make a similar definition for preserving normality with respect to the regular continued fraction expansion, $\beta$-expansions, the L\"uroth series expansion, etc.

Several authors have studied $b$-normality preserving functions.  Some $b$-normality preserving functions naturally arise in  H. Furstenberg's work on disjointness in ergodic theory\cite{FurstenbergDisjoint}.  V. N. Agafonov \cite{AgafonovNormal}, T. Kamae \cite{Kamae}, T. Kamae and B. Weiss \cite{KamaeWeiss}, and W. Merkle and J. Reimann \cite{MerkleReimann} studied $b$-normality preserving selection rules.

For a real number $r$, define real functions $\pi_r$ and $\sigma_r$ by $\pi_r(x)=rx$ and $\sigma_r(x)=r+x$.  In 1949 D. D. Wall proved in his Ph.D. thesis \cite{Wall} that for non-zero rational $r$ the function $\pi_r$ is $b$-normality preserving for all $b$ and that the function $\sigma_r$ is $b$-normality preserving functions for all $b$ whenever $r$ is rational.  These results were also independently proven by K. T. Chang in 1976 \cite{ChangNormal}.  D. D. Wall's method relies on the well known characterization that a real number $x$ is normal in base $b$ if and only if the sequence $(b^nx)$ is uniformly distributed mod $1$ that he also proved in his Ph.D. thesis.

D. Doty, J. H. Lutz, and S. Nandakumar took a substantially different approach from D. D. Wall and strengthened his result.  They proved in \cite{LutzNormalityPreserves} that for every real number $x$ and every non-zero rational number $r$  the $b$-ary expansions of $x, \pi_r(x),$ and $\sigma_r(x)$ all have the same finite-state dimension and the same finite-state strong dimension.  It follows that $\pi_r$ and $\sigma_r$ preserve $b$-normality.  It should be noted that their proof uses different methods from those used by D. D. Wall and is unlikely to be proven using similar machinery.

C. Aistleitner generalized D. D. Wall's result on $\sigma_r$.  Suppose that $q$ is a rational number and that the digits of the $b$-ary expansion of $z$ are non-zero on a set of indices of density zero.  In \cite{AistleitnerNormalPreserves} he proved that the function $\sigma_{qz}$ is $b$-normality preserving.  We will show as a consequence of \reft{densityzerochange} that C. Aistleitner's result does not generalize to at least one notion of normality for some of the Cantor series expansions.

There are still many open questions relating to the functions $\pi_r$ and $\sigma_r$.  For example, M. Mend\'{e}s France asked in \cite{X} if the function $\pi_r$ preserves simple normality with respect to the regular continued fraction for every non-zero rational $r$.  The authors are unaware of any theorems that state that either $\pi_r$ or $\sigma_r$ preserve any other form of normality than $b$-normality.

We will focus on the normality preserving properties of $\pi_r$ for the $Q$-Cantor series expansion as well as two other related functions.  We will show that while $\pi_r$ is $Q$-distribution normality preserving for every non-zero integer $r$, the set of $x$ where $\pi_r(x)$ is not $Q$-distribution normal has full Hausdorff dimension whenever $r \in \mathbb{Q} \backslash \mathbb{Z}$ and $Q$ is infinite in limit. Our main theorem will show that the function $\pi_r$ is so far from preserving $Q$-normality that there exist basic sequences $Q$ and real numbers $x$ that are $Q$-normal and $Q$-distribution normal where $\pi_r(x)$ is not $Q$-normal for every integer $r \geq 2$.  In the sequel to this paper \cite{AireyManceVandehey}, the authors and J. Vandehey prove that for a class of basic sequences $Q$, the set of real numbers $x$ where $\pi_r(x)$ is $Q$-normal for all non-zero rationals $r$ but where $x$ is not $Q$-distribution normal has full Hausdorff dimension.


\section{Cantor series expansions}

The study of normal numbers and other statistical properties of real numbers with respect to large classes of Cantor series expansions was  first done by P. Erd\H{o}s and A. R\'{e}nyi in \cite{ErdosRenyiConvergent} and \cite{ErdosRenyiFurther} and by A. R\'{e}nyi in \cite{RenyiProbability}, \cite{Renyi}, and \cite{RenyiSurvey} and by P. Tur\'{a}n in \cite{Turan}.

The $Q$-Cantor series expansions, first studied by G. Cantor in \cite{Cantor},
are a natural generalization of the $b$-ary expansions.\footnote{G. Cantor's motivation to study the Cantor series expansions was to extend the well known proof of the irrationality of the number $e=\sum 1/n!$ to a larger class of numbers.  Results along these lines may be found in the monograph of J. Galambos \cite{Galambos}. } 
Let $\mathbb{N}_k:=\mathbb{Z} \cap [k,\infty)$.  If $Q \in \NN$, then we say that $Q$ is a {\it basic sequence}.
Given a basic sequence $Q=(q_n)_{n=1}^{\infty}$, the {\it $Q$-Cantor series expansion} of a real number $x$  is the (unique)\footnote{Uniqueness can be proven in the same way as for the $b$-ary expansions.} expansion of the form
\begin{equation} \labeq{cseries}
x=E_0+\sum_{n=1}^{\infty} \frac {E_n} {q_1 q_2 \cdots q_n}
\end{equation}
where $E_0=\floor{x}$ and $E_n$ is in $\{0,1,\ldots,q_n-1\}$ for $n\geq 1$ with $E_n \neq q_n-1$ infinitely often. We abbreviate \refeq{cseries} with the notation $x=E_0.E_1E_2E_3\ldots$ w.r.t. $Q$.

A {\it block} is an ordered tuple of non-negative integers, a {\it block of length $k$} is an ordered $k$-tuple of integers, and {\it block of length $k$ in base $b$} is an ordered $k$-tuple of integers in $\{0,1,\ldots,b-1\}$.

Let
$$
Q_n^{(k)}:=\sum_{j=1}^n \frac {1} {q_j q_{j+1} \cdots q_{j+k-1}} \hbox{ and }  T_{Q,n}(x):=\left(\prod_{j=1}^n q_j\right) x \pmod{1}.
$$
A. R\'enyi \cite{Renyi} defined a real number $x$ to be {\it normal} with respect to $Q$ if for all blocks $B$ of length $1$,
\begin{equation}\labeq{rnormal}
\lim_{n \rightarrow \infty} \frac {N_n^Q (B,x)} {Q_n^{(1)}}=1.
\end{equation}
If $q_n=b$ for all $n$ and we restrict $B$ to consist of only digits less than $b$, then \refeq{rnormal} is equivalent to {\it simple normality in base $b$}, but not equivalent to {\it normality in base $b$}. 
A basic sequence $Q$ is {\it $k$-divergent} if
$\lim_{n \rightarrow \infty} Q_n^{(k)}=\infty$ and {\it fully divergent} if $Q$ is $k$-divergent for all $k$. 
A basic sequence $Q$ is {\it infinite in limit} if $q_n \rightarrow \infty$.

\begin{definition}\labd{1.7} A real number $x$  is {\it $Q$-normal of order $k$} if for all blocks $B$ of length $k$,
$$
\lim_{n \rightarrow \infty} \frac {N_n^Q (B,x)} {Q_n^{(k)}}=1.
$$
We let $\Nk{Q}{k}$ be the set of numbers that are $Q$-normal of order $k$.  The real number $x$ is {\it $Q$-normal} if
$x \in \NQ := \bigcap_{k=1}^{\infty} \Nk{Q}{k}.$
A real number~$x$ is {\it $Q$-distribution normal} if
the sequence $(T_{Q,n}(x))_{n=0}^\infty$ is uniformly distributed mod $1$.  Let $\DNQ$ be the set of $Q$-distribution normal numbers.
\end{definition}

It  follows from a well known result of H. Weyl \cite{Weyl2,Weyl4} that $\DNQ$ is a set of full Lebesgue measure for every basic sequence $Q$. We will need the following result of the second author \cite{Mance4} later in this paper.

\begin{thrm}\labt{measure}\footnote{Early work in this direction has been done by A. R\'enyi \cite{Renyi}, T.  \u{S}al\'at \cite{Salat4}, and F. Schweiger~\cite{SchweigerCantor}.}
Suppose that $Q$ is infinite in limit.  Then $\Nk{Q}{k}$ (resp. $\NQ$) is of full measure if and only if $Q$ is $k$-divergent (resp. fully divergent). 
\end{thrm}

Note that in base~$b$, where $q_n=b$ for all $n$,
 the corresponding notions of $Q$-normality and $Q$-distribution normality are equivalent. This equivalence
is fundamental in the study of normality in base $b$. 

Another definition of normality, $Q$-ratio normality, has also been studied.  We do not introduce this notion here as this set contains the set of $Q$-normal numbers and all results in this paper that hold for $Q$-normal numbers also hold for $Q$-ratio normal numbers.
The complete containment relation between the sets of these normal numbers and pair-wise intersections thereof is proven in \cite{ppq1}.  The Hausdorff dimensions of difference sets such as $\RDN$ are computed in \cite{AireyManceHDDifference}.

A surprising property of $Q$-normality of order $k$ is that we may not conclude that $\Nk{Q}{k} \subseteq \Nk{Q}{j}$ for all $j <k$ like we may for the $b$-ary expansions.  In fact, it was shown in \cite{ppq2} that for every $k$ there exists a basic sequence $Q$ and a real number $x$ such that $\Nk{Q}{k} \backslash \bigcup_{j=1}^{k-1} \Nk{Q}{j}$ is non-empty.  Thus, rather than showing that some functions do not preserve $Q$-normality of order $k$, we will show that they do not preserve $Q$-normality of any order.  We will always demonstrate numbers not $Q$-normal of any order that either have at most finitely many copies of the digit $0$ or the digit $1$ in their $Q$-Cantor series expansion.

\section{Results}

We note the following theorem which may be stated in terms of $Q$-normality preserving functions.  Instead we present it in its current form for simplicity.

\begin{thrm}\labt{densityzerochange}
Suppose that $Q$ is infinite in limit and that $x=E_0.E_1E_2\cdots\wrt{Q}$ is $Q$-normal of order $k$.  Then there exists a real number $y=F_0.F_1F_2\cdots\wrt{Q}$ where $E_n \neq F_n$ on a set of density zero and $y \notin \bigcup_{j=1}^\infty \Nk{Q}{j}$.
\end{thrm}

\reft{densityzerochange}  may be proven by changing all the digits of a $Q$-normal number that are equal to $0$ to $1$.
This shows that C. Aistleitner's result does not generalize to $Q$-normality for all $Q$-Cantor series expansions by letting $q=1$ and by letting $z=0.G_1G_2\cdots\wrt{Q}$ where $G_n$ is equal to $1$ along each of these indices and $0$ otherwise and setting $y=x+z$.


\begin{thrm}\labt{changedistnormal}
Suppose that $Q$ is a basic sequence 
and that $x=E_0.E_1E_2\cdots\wrt{Q}$ is $Q$-distribution normal.  If $y=F_0.F_1F_2\cdots\wrt{Q}$ and
\begin{equation}\labeq{condition}
\lim_{N \to \infty} \frac {1} {N} \sum_{n=1}^N \frac {|E_n-F_n|+1} {q_n} =0,
\end{equation}
then $y \in \DNQ$.
\end{thrm}

\begin{cor}
Suppose that $Q$ satisfies
$$
\lim_{N \to \infty} \frac {1} {N} \sum_{n=1}^N \frac {1} {q_n} =0
$$
 and that $x=E_0.E_1E_2\cdots\wrt{Q}$ is $Q$-distribution normal.  If $y=F_0.F_1F_2\cdots\wrt{Q}$ and
\begin{equation}\labeq{condition}
\lim_{n \to \infty} \frac {|E_n-F_n|} {q_n} =0,
\end{equation}
then $y \in \DNQ$.
\end{cor}

Together, \reft{densityzerochange} and \reft{changedistnormal} suggest that $Q$-distribution normality is a far more robust notion than $Q$-normality.  We wish to give the following example demonstrating that distribution normality is not preserved in all bases by rational multiplication.

\begin{example}
Define the sequences $E=(E_n)$ and $Q=(q_n)$ by
\begin{align*}
E=(0,0,2,0,2,4,0,2,4,6,\cdots);\\
Q=(2,4,4,6,6,6,8,8,8,8,\cdots).
\end{align*}
Set $x=\formalsum$.  Then $x \in \DNQ$, but $\pi_{1/2}(x) \notin \DNQ$.
\end{example}
However, a much stronger result holds:

\begin{thrm}\labt{distpres}
For all basic sequences $Q$, $Q$-distribution normality is preserved by non-zero integer multiplication.  If $Q$ is infinite in limit and $r \in \mathbb{Q} \backslash \mathbb{Z}$, then 
\begin{align}\labeq{measrDNQ}
\lambda\left(\br{x \in \DNQ : \pi_r(x) \notin \DNQ }\right)=0;\\
\labeq{HDDNQ}
\dimh{\br{x \in \DNQ : \pi_r(x) \notin \DNQ }}=1.
\end{align}
\end{thrm}

We wish to define an equivalence relation $\sim$ on the set of basic sequences as follows.
If $P=(p_n)$ and $Q=(q_n)$ are basic sequences then we write $P \sim Q$ if $p_n \neq q_n$ on a set of density zero.

Suppose that $Q$ that is infinite in limit and fully divergent.  A nonempty subset $S_Q \subseteq \NQ \backslash \DNQ$ was shown to exist in Theorem 3.12 in \cite{ppq1}.  The members of $S_Q$ have the following property.  If $x \in S_Q$, then for any integers $n \geq 2$, the real number $\pi_n(x)$ is not $Q$-distribution normal and not $Q$-normal of any order.  Since $Q$-distribution normality is preserved by integer multiplication it is natural to ask if there are any basic sequences $Q$ and real numbers $x$ that are $Q$ normal and $Q$-distribution normal and such that for any integers $n \geq 2$ the number $\pi_n(y)$ are not $Q$-normal of any order.  The following theorem answers this question.

\begin{thrm}\labt{main}
Let $k\in \mathbb{N}$.  If $P=(p_n)$ is eventually non-decreasing, infinite in limit, and $k$-divergent (resp. fully divergent), then there exists a basic sequence $Q=(q_n)$ and a real number $x$ where the following hold.
\begin{enumerate}
\item The basic sequence $Q$ is infinite in limit,  $k$-divergent (resp. fully divergent), and $P \sim Q$.
\item The real number $x$ is $Q$-normal of all orders $1$ through $k$ (resp. $Q$-normal) and $Q$-distribution normal.
\item For every integer $n \geq 2$, the real number $\pi_n(x)$ is not $Q$-normal of any order.
\end{enumerate}
\end{thrm}

For a sequence of real numbers $X = (x_n)$ with $x_n \in [0,1)$ and an interval $I \subseteq [0,1]$,  define $A_n(I,X) = \# \{i\leq n: x_i \in I \}$.
We will first prove \reft{changedistnormal}.  To do this we will need the following standard definition and lemma that we quote from \cite{KuN}.

\begin{definition}
Let $X = \pr{x_1, \cdots , x_N}$ be a finite sequence of real numbers. The number $$D_N = D_N(X) = \sup_{0 \leq \alpha \leq \beta \leq 1} \left | \frac{A_N([\alpha, \beta), X)}{N} - (\beta - \alpha) \right |$$ is called the {\it discrepancy} of the sequence $X$.
\end{definition}

It is well known that a sequence $X$ is uniformly distributed mod $1$ if and only if $D_N(X) \to 0$.

\begin{lem}\labl{DiscKuN}
Let $x_1, x_2, \cdots, x_N$ and $y_1, y_2, \cdots, y_N$ be two finite sequences in $[0,1)$.  Suppose $\epsilon_1, \epsilon_2, \cdots, \epsilon_N$ are non-negative numbers such that $|x_n-y_n| \leq \epsilon_n$ for $1 \leq n \leq N$.  Then, for any $\epsilon \geq 0$, we have
$$
|D_N(x_1,\cdots,x_N)-D_N(y_1,\cdots,y_N)| \leq 2\epsilon+\frac {\overline{N}(\epsilon)}{N},
$$
where $\overline{N}(\epsilon)$ denotes the number of $n$, $1 \leq n \leq N$, such that $\epsilon_n>\epsilon$.
\end{lem}

\begin{proof}[Proof of \reft{changedistnormal}]
Set $\epsilon_i = \frac{1}{i}$. Define the sets 
$$
S(\epsilon_i) = \br{ i : \frac{|E_i-F_i|+1}{q_i} \geq \epsilon_i}
$$ 
and the sequence 
$$
N_i = \min \br{ n : \frac{\# S(\epsilon_i) \cap \{ 1, \cdots, j \}}{j} < \frac{1}{i}, \forall j\geq n}.
$$
 Note that $N_i$ is defined since the density of the sets $S(\epsilon_i)$ must be $0$ by \refeq{condition}. Set $x_i = T_{Q,i}(x)$ and $y_i = T_{Q,i}(y)$. Note that $|x_i-y_i| \leq \frac{|E_i-F_i|+1}{q_i}$.
Then for any $n > N_i$ we have by \refl{DiscKuN} that 
$$
| D_n(x_1, \cdots, x_n) - D_n(y_1, \cdots, y_n) | \leq 2 \epsilon_i + \frac{\#S(\epsilon_i) \cap \br{1, \cdots , N}}{N} < \frac{3}{i}.
$$ 
Thus, $\lim_{n \to \infty} | D_n(x_1, \cdots, x_n) - D_n(y_1, \cdots, y_n)| = 0$. This implies that $y$ is $Q$-distribution normal if and only if $x$ is.
\end{proof}

For the remainder of this paper let $E_{R,i}(\xi)$ be the $i$th digit of the $R$-Cantor series expansion of $\xi$.

\begin{proof}[Proof of \reft{distpres}]
The first part is trivial as any integer multiple of a uniformly distributed sequence is uniformly distributed. It is 
well known that locally Lipschitz functions preserve null sets. Clearly $\pi_r$ is locally Lipschitz,  so \refeq{measrDNQ} holds.

Let $r=a/b \in \mathbb{Q} \backslash \mathbb{Z}$ for relatively prime integers $a$ and $b$.  
We now wish to show \refeq{HDDNQ}.  Let $(x_n)$ be a sequence of real numbers that is uniformly distributed mod 1. Define 
$$
C(n) = \max \br{m \in b\mathbb{Z} : \frac{m}{q_n} \leq x_n}.
$$ 
Set 
$$
f(n) = \frac{\min \br{\log q_n, \log q_1 q_2 \cdots q_{n-1}}}{\log q_n} \hbox{ and }\omega(n) = \floor{\frac{q_n^{1-f(n)}}{b}}.
$$
 Note that 
$
\lim_{n \to \infty} \frac{\omega(n)}{q_n} = 0
$ 
since 
$$
\lim_{n \to \infty} f(n) \log(q_n) \to \infty.
$$ 
 Define the intervals 
$$
V_n = \left[C(n) - \frac{b}{2} \omega(n), C(n)+\frac{b}{2} \omega(n)\right] \cap b\mathbb{Z}.
$$
Consider the set 
$$
\Phi_{Q,b} = \{ x=0.E_1E_2\cdots\wrt{Q} : E_n \in V(n) \}.
$$ 
Note that for any $x \in \Phi_{Q,b}$, we have 
that 
$$
\lim_{n \to \infty} \pr{\frac{E_n}{q_n} - x_n} = 0
$$ 
since $\frac{\omega(n)}{q_n} \to 0$. This implies that $\pr{\frac{E_n}
{q_n}}$ is uniformly distributed mod 1, so the seqeunce $\pr{T_{Q,n}(x)}$ is as well. Therefore $\Phi_{Q,b} \subseteq \DN{Q}$. Furthermore,  
every digit of $x \in \Phi_{Q,b}$ is divisible by $b$. Thus 
$$
\frac{E_{Q,n}\pr{\frac{1}{b} x}}{q_n} \in \left[0,\frac{1}{b}\right)
$$ 
which implies that 
$$
T_{Q,n}\pr{\frac{1}{b} x} \in \left[0, \frac{1}{b} + \frac{1}{q_n} \right).
$$ 
We have that if 
$$
T_{Q,n}\pr{\frac{a}{b}x} \in \left[\frac{\floor{a/b}}{a}, \frac{1}
{b}+\frac{1}{q_n}\right),$$
 then 
$$
T_{Q,n}\pr{\frac{a}{b} x} \in \left[0, \frac{a}{b}+\frac{a}{q_n}\right).
$$ 
Let $c \equiv a \mod b$. Thus
$$
\lim_{n \to \infty} \frac{A_n\pr{[0,\frac{c}{b}), \pr{T_{Q,n}(\pi_r(x))}}}{n} = \frac{\pr{\frac{1}{b} - \frac{\floor{a/b}}{a}}\floor{a/b}}{\frac{1}{b}} = \floor{\frac{a}{b}} \frac{c}{a} \neq \frac{c}{b},
$$
so $\pi_r(x) \notin \DNQ$.

Following the notation of \cite{FengWenWu}, $\Phi_{Q,b}$ is a homogeneous Moran set with $c_k = \frac{1}{q_k}$ and $n_k = \omega(k)$. By Theorem 2.1 in \cite{FengWenWu}, we have that 
\begin{align*}
\dimh{\Phi_{Q,b}} &
\geq \liminf_{k \to \infty} \frac{\log n_1 n_2 \cdots n_k}{-\log c_1 c_2 \cdots c_{k+1} n_{k+1}} \\
&= \liminf_{k \to \infty} \frac{\log \omega(1) \omega(2) \cdots \omega(k)}{- \log\pr{\frac{1}{q_1} \frac{1}{q_2} \cdots \frac{1}{q_{k+1}} \omega(k+1)}}\\
&= \liminf_{k \to \infty} \frac{\log \pr{q_1^{1-f(i)} q_2^{1-f(2)} \cdots q_k^{1-f(k)}} - k \log b}{\log q_1 q_2 \cdots q_k - f(k+1)\log q_{k+1}} \\
&= \liminf_{k \to \infty} \frac{\log q_1 \cdots q_k}{\log q_1 \cdots q_k - f(k+1) \log q_{k+1}} \\
& = \liminf_{k \to \infty} \frac{1}{1 - \frac{f(k+1) \log q_{k+1}}{\log q_1 \cdots q_{k}}} = 1.
\end{align*}
\end{proof}

We may now turn our attention to \reft{main}.  
Let $(P,Q) \in \NNC$ and suppose that $x=E_0.E_1E_2 \cdots$ w.r.t. $P$.  We define $\ppq:\mathbb{R} \to [0,1]$ by
$$
\ppq(x):=\sum_{n=1}^\infty \frac {\min(E_n,q_n-1)} {\q{n}}.
$$
The following theorem of \cite{ppq1} will be critical in proving \reft{main}.
\begin{thrm}\labt{mainpsi}
Suppose  that $Q_1=(q_{1,n}), Q_2=(q_{2,n}),\cdots, Q_j=(q_{j,n})$ are basic sequences and infinite in limit.
  If $x=E_0.E_1E_2\cdots$ w.r.t $Q_1$ satisfies
$E_n<\min_{2 \leq r \leq j} (q_{r,n}-1)$ for infinitely many $n$, then for every block~$B$
$$
N_{n}^{Q_j}\left(B,\left(\psi_{Q_{j-1},Q_j} \circ \psi_{Q_{j-2},Q_{j-1}} \circ \cdots \circ \psi_{Q_1,Q_2}\right)(x)\right)
=N_{n}^{Q_1}(B,x)+O(1).
$$
\end{thrm}

Here we state a theorem of C. T. Rajagopal \cite{RajagopalWeightedDensity}.

\begin{thrm}\labt{RajWeight}
Let $(w_n)$ and $(s_n)$ be sequences of positive real numbers such that $\sum_{n=1}^\infty w_n = \sum_{n=1}^\infty s_n = \infty$. If $\frac{w_n}{s_n}$ is non-increasing, then for every sequence $(a_n)$ of real numbers, we have that
\begin{align*}
&\liminf_{n \to \infty} \frac{\sum_{k=1}^n s_n a_n}{\sum_{k=1}^n s_n} \leq \liminf_{n \to \infty} \frac{\sum_{k=1}^n w_n a_n}{\sum_{k=1}^n w_n} \\
&\leq \limsup_{n \to \infty}  \frac{\sum_{k=1}^n w_n a_n}{\sum_{k=1}^n w_n} \leq \limsup_{n \to \infty} \frac{\sum_{k=1}^n s_n a_n}{\sum_{k=1}^n s_n}.
\end{align*}
\end{thrm}

We also note the following basic lemma.

\begin{lem}\labl{tcorr}
Let $L$ be a real number and $(a_n)_{n=1}^\infty$ and $(b_n)_{n=1}^\infty$ be two sequences of positive real numbers such that
$$
\sum_{n=1}^{\infty} b_n=\infty \hbox{ and } \lim_{n \to \infty} \frac {a_n} {b_n}=L.
$$
Then
$$
\lim_{n \to \infty} \frac {a_1+a_2+\ldots+a_n} {b_1+b_2+\ldots+b_n}=L.
$$
\end{lem}

\begin{proof}[Proof of \reft{main}]
Let $P$ be a basic sequence that is infinite in limit and k-divergent. The proof of the statement for when $P$ is fully divergent will follow similarly.
Let $y$ be a real number that is  $P$-distribution normal such that $ny$ is $P$-normal for all natural numbers $n$. Note that such a $y$ exists by \reft{measure}. 

Define the sequence $\{\ell_n\}$ as follows:
\begin{align*}
&\ell_1 = 1; \\
& \ell_{i} = \min \br{t : {\sum_{j=1}^i N_n^P(1,y)} < \frac{n}{i} ,\forall n \geq t }.
\end{align*}
Define $M(n) = \min\br {c : l_c < n }$.
We have that $M(n)$ tends to infinity since 
$$
\frac{\sum_{j=1}^i N_n^P(1,y)}{\sum_{j=1}^i P_n^{(1)}} \to 1
$$
and $P$ is infinite in limit.
Furthermore, for any $i \leq M(N)$ and $n \geq N$ we have that $\frac{\sum_{j=1}^{M(N)} N_n^P(1,y)}{n} < \frac{1}{M(N)}$.

Construct $Q=(q_n)$ as follows:
If there is a $i \in \br{1 , \cdots , M(n)}$ such that $E_{P,n}(\pi_i(y)) = 1$, then set $q_n = M(n) p_n$ and $q_{n+1} = M(n) p_{n+1}$. Otherwise, set $q_n = p_n$. If both $E_{P,n}(\pi_i(y)) = 1$ and $E_{P,n-1}(\pi_j(y)) =1$ for some $i, j \leq M(n)$, put $q_n = M(n) p_n$.
Put $x = \ppq(y)$. Then at position $n$, we have that $E_{Q,n}(\pi_m(x)) \neq 1$ when $m \in \{2, 3,\cdots,M(n)\}$. Note that for $E_{Q,n}(\pi_m(x)) =1$, we must have that 
\begin{align*}
\frac{m E_{P,n}(y)}{q_n} + \frac{m E_{P,n+1}(y)}{q_n q_{n+1}} + \cdots &\in \left[\frac{1}{q_n}, \frac{2}{q_n}\right) \hbox{ or }\\
\frac{m E_{P,n}(y)}{ M(n) p_n} + \frac{m E_{P,n+1}(y)}{M(n)^2 p_n p_{n+1}} + \cdots &\in \left[ \frac{1}{M(n) p_n}, \frac{2}{M(n) p_n}\right).
\end{align*} 
But for this to happen, we must have that $E_{P,n}(y) = 0$, and $E_{P,i}(y) = p_{i}-1$ for all $i> n$. 
However this can not happen since $y=E_{P,0}(y).E_{P,1}(y)E_{P,2}(y)\cdots\wrt{P}$ is the $P$-Cantor series expansion of $y$.
Thus we must have that $E_{Q,n}(\pi_m(x)) \neq 1$.
Then for all  $m \in \mathbb{N}_2$, we have that $\lim_{n \to \infty} N_n^Q(1,\pi_m(x))$ is finite. Therefore $\pi_m(x)$ is not $Q$-normal of any order for all $m \in \mathbb{N}_2$.

Let $A \subseteq \mathbb{N}$ be the set of indices where $q_n \neq p_n$. Since 
$$
\frac{\sum_{j=1}^{M(N)} N_n^P(1,y)}{n} < \frac{1}{M(N)},
$$
we have that this set has density zero.
Note that $\pr{\frac{E_{Q,i}(x)}{q_i}}$ differs from $\pr{\frac{E_{P,i}(x)}{p_i}}$ if and only if $n \in A$. Since $\pr{\frac{E_{P,i}(y)}{p_i}}$ is uniformly distributed mod 1, we have that $\pr{\frac{E_{Q,i}(x)}{q_i}}$ is uniformly distributed mod 1 by \refl{DiscKuN}. Since $Q$ is infinite in limit, we have that $x$ is $Q$-distribution normal.

Fix some $j \in \{1,2,\cdots,k\}$. Note that $$\sum_{i=1}^n \frac{\chi_{\mathbb{N}\backslash A}(i)}{p_i \cdots p_{i+j-1}} \leq Q_n^{(j)} \leq P_n^{(j)}.$$
Using the notation of \reft{RajWeight}, set $a_n= \chi_A(n)$, $w_n = \frac{1}{p_i \cdots p_{i+j-1}}$, and $s_n = 1$. Since $P$ is non-decreasing and $k$-divergent, we can apply \reft{RajWeight}. Thus  we have
$$
\limsup_{n \to \infty} \frac{\sum_{i=1}^n \frac{\chi_A(n)}{p_i \cdots p_{i+j-1}}}{P_n^{(j)}} = \limsup_{n \to \infty} \frac{\# A \cap \{1, \cdots, n\}}{n} = 0,
$$
so $\lim_{n \to \infty} \frac{Q_n^{(j)}}{P_n^{(j)}} = 1.$
Since $N_n^P(B,y) = N_n^Q(B,x)$, we have that $x$ is $Q$-normal of orders $1, 2, \cdots, k$. 
\end{proof}

\section{Further problems}
\reft{main} suggests some natural further investigations.
\begin{problem}
Can the condition that $Q$ is monotone be weakened or replaced by other conditions in \reft{main}?
\end{problem}

We consider the following two conditions on a basic sequence $Q$ and a real number~$x$.
\begin{align}\labeq{firstcondition}
&x \in \DNQ \cap \bigcap_{j=1}^k \Nk{Q}{j} &\hbox{ and }nx \notin \bigcup_{j=1}^\infty \Nk{Q}{j} \ \forall n \in \mathbb{N}_2;\\
\labeq{secondcondition}
&x \in \DNQ \cap \NQ &\hbox{ and }nx \notin \bigcup_{j=1}^\infty \Nk{Q}{j} \ \forall n \in \mathbb{N}_2.
\end{align}
\begin{problem}
Is it true that for all $Q$ that are $k$-divergent and infinite in limit that there exists a real number $x$ satisfying \refeq{firstcondition}?  If $Q$ is fully divergent and infinite in limit must there exist an $x$ that satisfies \refeq{secondcondition}?  If not, what must we assume about $Q$?
\end{problem}
We note that the use of \reft{measure} in the proof of \reft{main} means we have not given any explicit examples of the basic sequence or real number $x$ mentioned in \reft{main}.  
There exist some basic sequences $Q$ where the set $\DNQ$ does not contain any computable real numbers.  See \cite{BerosBerosComputable}.  
Thus, it is reasonable to ask the following question.
\begin{problem}
Give an example of a computable basic sequence and a computable real number $x$ that satisfies the conditions \refeq{firstcondition} or \refeq{secondcondition}.  Can this be done for every computable basic sequence $Q$?
\end{problem}
The authors strongly believe that sets of numbers that satisfy conditions \refeq{firstcondition} or \refeq{secondcondition} must have full Hausdorff dimension.  Thus we ask
\begin{problem}
For $k$-divergent $Q$ that are infinite in limit, compute
$$
\dimh{\br{x : \hbox{condition \refeq{firstcondition} holds}}}.
$$
If $Q$ is fully divergent and infinite in limit, compute
$$
\dimh{\br{x : \hbox{condition \refeq{secondcondition} holds}}}.
$$
\end{problem}

\bibliographystyle{amsplain}

\providecommand{\bysame}{\leavevmode\hbox to3em{\hrulefill}\thinspace}
\providecommand{\MR}{\relax\ifhmode\unskip\space\fi MR }
\providecommand{\MRhref}[2]{%
  \href{http://www.ams.org/mathscinet-getitem?mr=#1}{#2}
}
\providecommand{\href}[2]{#2}

\end{document}